\documentclass[a4paper]{amsart}

\usepackage[hmargin=3.2cm,vmargin=3.2cm,a4paper,centering,twoside]{geometry}

\usepackage[T1]{fontenc}
\usepackage{ae,aecompl}
\usepackage{verbatim}
\usepackage{ifpdf}
\usepackage{booktabs}

\usepackage{hyperref}
\hypersetup{
    colorlinks=false,   
    citecolor=green,    
    filecolor=magenta,  
    linkcolor=red,      
    urlcolor=blue       
}

\makeatletter
\newcommand\org@hypertarget{}
\let\org@hypertarget\hypertarget
\renewcommand\hypertarget[2]{%
    \Hy@raisedlink{\org@hypertarget{#1}{}}#2%
} 
\makeatother

\ifpdf
    \hypersetup{
        pdftitle={Deformation method},
        pdfauthor={Sebastian Pancratz, Jan Tuitman},
        pdfsubject={Computational Number Theory},
        bookmarks=true,
        bookmarksnumbered=true,
        unicode=true,
        pdfstartview={FitH},
        pdfpagemode={UseOutlines}
    }
\fi

\makeatletter
\def\blfootnote{\xdef\@thefnmark{}\@footnotetext}
\makeatother

\usepackage[section]{algorithm}
\usepackage[noend]{algpseudocode}

\usepackage{natbib}

\bibpunct{[}{]}{,}{n}{}{}

\usepackage{url}

\makeatletter
\def\url@leostyle{%
  \@ifundefined{selectfont}{\def\UrlFont{\sf}}{\def\UrlFont{\small\ttfamily}}}
\makeatother
\usepackage{paralist}

\usepackage{amsmath,amsthm,amscd,amsfonts,amssymb}
\usepackage{cases}
\usepackage[all]{xy}

\allowdisplaybreaks[4]
\numberwithin{equation}{section}

\newcommand{\NN}{\mathbf{N}} 
\newcommand{\ZZ}{\mathbf{Z}} 
\newcommand{\QQ}{\mathbf{Q}} 
\newcommand{\FF}{\mathbf{F}} 
\newcommand{\PP}{\mathbf{P}} 

\DeclareMathOperator{\Gal}{Gal}          
\DeclareMathOperator{\ord}{ord}          
\DeclareMathOperator{\Frob}{F}           

\providecommand{\HdR}{H_{\text{dR}}}    
\providecommand{\Hrig}{H_{\text{rig}}}  

\providecommand{\cB}{\mathcal{B}} 

\providecommand{\BigOh}{O}          
\providecommand{\SoftOh}{\tilde{O}} 

\theoremstyle{definition}

\newtheorem{thm}{Theorem}[section]

\newtheorem{defn}[thm]{Definition}

\newtheorem{rem}[thm]{Remark}

\makeatletter

\newcommand{\Rmnum}[1]{\expandafter\@slowromancap\romannumeral #1@}
\makeatother

\title{Computing zeta functions of generic projective hypersurfaces in larger characteristic}
\author{Jan Tuitman}
\address{KU Leuven,
         Departement Wiskunde,
         Celestijnenlaan 200B,
         3001 Leuven,
         Belgium}
\email{jan.tuitman@kuleuven.be}

\begin{document}

\begin{abstract}{We give an improvement of the deformation method for computing the zeta function 
of a generic projective hypersurface over a finite field of characteristic~$p$ that reduces the dependence 
of the complexity on~$p$ to $\SoftOh(p^{1/2})$ while remaining polynomial in the other input parameters.} 
\end{abstract}

\maketitle 


\section{Introduction}

Let $\FF_q$ be a finite field of characteristic~$p$ and cardinality $q=p^a$. Suppose that $X$ is
a projective hypersurface of degree $d$ in projective $n$-space $\PP^n_{\FF_q}$. For all $i \in \NN$
we can then count the number of points of $X$ with values in $\FF_{q^i}$. Note that $X$ is 
defined by some homogeneous polynomial of degree~$d$ in $\FF_q[x_0,\ldots,x_n]$ and 
$X(\FF_{q^i})$ is just the number of zeros of this polynomial in projective $n$-space over $\FF_{q^i}$. 
These numbers of points can be used to define a generating series that is called the 
zeta function of $X$:
\[
Z(X,T) = \exp \left( \sum_{i=1}^{\infty} \lvert X(\FF_{q^i}) \rvert \frac{T^i}{i} \right). 
\]

It is well known that $Z(X,T)$ is the quotient of two polynomials with integer coefficients, so
can be given by a finite amount of data. An obvious question is therefore whether it can be
computed effectively. This is clearly the case, since there are well known bounds for the degrees
of the numerator and denominator of $Z(X,T)$, so that one is reduced to computing a finite number
of the $X(\FF_{q^i})$, which can be determined by naive counting. A more interesting problem is
whether, and if so how, $Z(X,T)$ can be computed efficiently. 

Lauder and Wan~\citep{LauderWan2008} proposed an algorithm that computes $Z(X,T)$ in time $(p a d^n)^{\BigOh(n)}$ without
any additional assumptions on $X$. Abbott, Kedlaya and Roe~\citep{AbbottKedlayaRoe2006} proposed a different algorithm with 
roughly the same complexity (but a smaller constant) which assumes $X$ to be generic (at least
smooth). The reason for this assumption is that in contrast to Lauder and Wan, the algorithm of
Abbott, Kedlaya and Roe uses $p$-adic cohomology, which is better behaved in the smooth case.

Note that the input size of the problem, i.e. the bit-size of the defining polynomial of~$X$, is about $\log(p) a d^n$. Therefore, 
the algorithms from the previous paragraph are not polynomial time for multiple reasons: $p^n$ is not polynomial in $\log(p)$, 
$a^n$ is not polynomial in $a$ and $d^{n^2}$ is not polynomial in $d^n$. Only for fixed $p$ and $n$ (but varying $a$ and $d$) are 
these algorithms polynomial time.

Lauder~\citep{Lauder2004a,Lauder2004b} then introduced the deformation method, which computes $Z(X,T)$ in time $(p a d^n)^{\BigOh(1)}$, for
a generic hypersurface $X$. Note that this is a polynomial time algorithm for fixed~$p$. 
The main idea is to compute the $p$-adic cohomology of $X$ by deforming it to a diagonal hypersurface.  The 
author and Pancratz~\citep{PT} improved the deformation method further and showed that it can be made to run in
\begin{align} \label{eq:PTcomplexity}
&\mbox{time: }  \SoftOh\left (p a^{3} d^{n(\omega+4)} e^{n(\omega+1)} \right),
&\mbox{space: }& \SoftOh\left(pa^3 d^{5n} e^{3n}\right),
\end{align}
where $\omega$ denotes an exponent for matrix multiplication, $e$ is the basis of the natural logarithm and we use the $\SoftOh(-)$ 
notation that ignores logarithmic factors, i.e.\ $\SoftOh(f)$ denotes the class of functions that lie in $\BigOh(f \log^k(f))$ for some $k \in \NN$.

All algorithms mentioned so far have complexity at least linear in $p$. Note that this is typically the case for all $p$-adic
point counting algorithms e.g.~\citep{Kedlaya2001,cdv,pcc1,pcc2}. However, Harvey has improved this situation, first for hyperelliptic curves~\citep{HarveySqrtp} 
and then for general schemes~\citep{HarveyArithmeticSchemes}. One algorithm from~\citep{HarveyArithmeticSchemes} 
runs in
\begin{align*} \label{eq:harvey1}
\mbox{time: }  &\BigOh \left(p^{1/2} \log(p)^{2+\epsilon} 2^{8n^2+16n} a^{4n+4+\epsilon} n^{4n+4+\epsilon} (d+1)^{4n^2+7n+\epsilon} \right), \\ 
\mbox{space: } &\BigOh\left(p^{1/2} \log(p) 2^{4n^2+9n} a^{2n+3} n^{2n+2} d^{2n^2+4n} \right).
\end{align*}
Note that this algorithm does not require the hypersurface $X$ to be generic, since the use of $p$-adic cohomology is 
avoided. 

We see that Harvey's algorithm is better than 
the deformation method in terms of the dependence of the time complexity on~$p$, but quite a lot worse in terms of the 
dependence on everything else (i.e. $a,d,n$). In particular, Harvey's algorithm is not polynomial time for fixed~$p$
anymore, since the complexity is polynomial in $d^{n^2}$ instead of $d^n$ and in $a^n$ instead of $a$. The goal of this 
paper is to modify the deformation method so that the dependence of its complexity on~$p$ becomes $\SoftOh(p^{1/2})$, 
while remaining polynomial in $a$, $d^n$. More precisely, we will prove the following theorem.

\begin{thm}
Let $X$ be a generic hypersurface of degree $d$ in projective space $\PP^n_{\FF_q}$ over
a finite field $\FF_q$ of characteristic~$p$ not dividing $d$ and cardinality $q=p^a$. Then the 
zeta function $Z(X,T)$ may be computed in
\begin{align*}
&\mbox{time: }  \SoftOh\left(p^{1/2} a^3 d^{n(2\omega+3)} e^{n(\omega+1)}\right), 
&\mbox{space: }& \SoftOh\left(p^{1/2} a^3 d^{n(2\omega+1)}e^{3n}\right).
\end{align*}
\end{thm}
\begin{proof}
This is a simplified version of Theorem~\ref{thm:complexity} below.
\end{proof}

What it means for a hypersurface to be generic is explained in Remark~\ref{rem:generic} below. Comparing our new 
algorithm to Harvey's we see that our exponents of $a$ and $d^n$ are constant instead of~$\BigOh(n)$ and 
smaller for any value of $n$. The dependence on~$p$ is more or less the same (probably exactly the same, with some more work). 
The only remaining advantages of Harvey's algorithm over ours are that it can be applied more generally and is 
conceptually simpler.


The author was supported by FWO Vlaanderen. We thank David Harvey for helpful discussions and the anonymous referees for
their useful suggestions.

\section{The deformation method}

In this section we briefly recall how the deformation method works. For more details the
reader should consult~\citep{PT}. We will keep the terminology and notation from 
that paper here as much as possible. 

Let $\FF_q$ be a finite field of characteristic~$p$ and cardinality $q=p^a$. Suppose that $X_1$ is
a projective hypersurface of degree~$d$ in projective $n$-space $\PP^n_{\FF_q}$ and let $U_1 = \PP^n_{\FF_q} \setminus X_1$ 
denote its complement. Moreover, let $\QQ_q$ be the unique unramified extension of $\QQ_p$ with residue 
field $\FF_q$, let $\ZZ_q$ be its ring of integers and $\sigma \in \Gal(\QQ_q/\QQ_p)$ the unique lift of the $p$-th
power map on $\FF_q$. We extend $\sigma$ to the standard $p$-th power 
Frobenius lift on $\mathbf{P}^1_{\QQ_q}$, i.e. the one that sends the standard coordinate~$t$ to $t^p$. 
Finally, let $\ord_p(-)$ denote the $p$-adic valuation on $\QQ_q$. 

For an algebraic variety $Y$ over $\FF_q$, let $\Hrig^{i}(Y)$ denote the rigid
cohomology spaces of $Y$. These are finite dimensional vector spaces over $\QQ_q$ 
that are contravariantly functorial in $Y$, and are equipped with a 
$\sigma$-semilinear action of the $p$-th power Frobenius map on $Y$ that we denote
by $\Frob_p$. For the construction and basic properties of these spaces we refer to~\citep{Berthelot1986}.

To compute the zeta function $Z(X_1,T)$ it is sufficient to compute 
the cohomology space $\Hrig^n(U_1)$ with the action of $\Frob_p$ by the following theorem.

\begin{thm} \label{thm:hypersurface} 
We have
\begin{equation*} \label{eq:formulazeta}
Z(X_1,T) = \frac{\chi(T)^{(-1)^n}}{(1 - T) (1 - qT) \dotsm (1 - q^{n-1}T)},
\end{equation*}
where 
$\chi(T) = \det \bigl( 1 - T (p^{-1} \Frob_{p})^a | \Hrig^n(U_1) \bigr) \in \ZZ[T]$
denotes the reverse characteristic polynomial of the action of $(p^{-1} \Frob_{p})^a$ 
on $\Hrig^n(U_1)$. Moreover, the polynomial $\chi(T)$ has degree 
\begin{equation*} \label{eq:formulab}
\frac{1}{d} \bigl((d-1)^{n+1} + (-1)^{n+1}(d-1) \bigr).
\end{equation*}
\end{thm}

\begin{proof}
\citep[Theorem 2.13]{PT}.
\end{proof}

The main idea of the deformation method is to take advantage of the way the rigid cohomology spaces
and the action of $\Frob_p$ vary in a family. 

\begin{defn}
Let $P_1 \in \ZZ_q[x_0,\ldots,x_n]$ be 
any homogeneous polynomial of degree $d$ not divisible by~$p$ that defines a hypersurface $\mathcal{X}_1 \subset \PP^n_{\ZZ_q}$ 
such that $\mathcal{X}_1 \otimes \FF_q \cong X$ and let $P_0 \in \ZZ_q[x_0,\ldots,x_n]$ be a diagonal polynomial
$P_0 = a_0 x_0^d + \ldots + a_n x_n^d$ with $a_1,\ldots,a_n \in \ZZ_p^{\times}$. Define $P \in \ZZ_q[t][x_0,\ldots,x_n]$ to be
\[P = (1-t)P_0 + tP_1.\] 
\end{defn}

The polynomial $P$ defines a family $\mathcal{X}/\mathcal{S}$ of smooth projective hypersurfaces  over some nonempty Zariski open subset
$\mathcal{S}$ of $\mathbf{P}^1_{\ZZ_q}$ such that the fibre $\mathcal{X}_0$ at~$t=0$ is the diagonal hypersurface defined by $P_0$ and 
the fibre $\mathcal{X}_1$ at~$t=1$ is isomorphic to $\mathcal{X}_1$. 
Let $\mathcal{U}/\mathcal{S}$ denote the complement
of $\mathcal{X}/\mathcal{S}$. We denote the generic fibres of $\mathcal{X},\mathcal{S}$, $\mathcal{U}$ by 
$\mathfrak{X}=\mathcal{X} \otimes \QQ_q$, $\mathfrak{S}=\mathcal{S} \otimes \QQ_q$, $\mathfrak{U}=\mathcal{U} \otimes \QQ_q$ and the special
fibres by $X=\mathcal{X} \otimes \FF_q$, $S = \mathcal{S} \otimes \FF_q$, $U = \mathcal{U} \otimes \FF_q$, 
respectively. 

\begin{thm}
The vector bundle $\HdR^n(\mathfrak{U}/\mathfrak{S})$ with its Gauss--Manin connection $\nabla$ 
admits a Frobenius structure $F$ with the following property: for any $\tau \in S(\FF_q)$ with Teichm\"uller lift 
$\hat{\tau} \in \mathcal{S}(\ZZ_q)$, we have
\[
(\Hrig^n(U_{\tau}),\Frob_p) \cong (\HdR^n(\mathfrak{U}/\mathfrak{S}),F)_{\hat{\tau}}
\] 
as $\QQ_{q}$-vector spaces with a $\sigma$-semilinear endomorphism. We will therefore denote this Frobenius structure 
on $\HdR^n(\mathfrak{U}/\mathfrak{S})$ by $\Frob_p$ as well.
\end{thm}

\begin{proof}
\citep[Theorem 2.10]{PT}.
\end{proof}

We can explicitly write down a basis for $\HdR^n(\mathfrak{U}/\mathfrak{S})$, at least generically.

\begin{defn}
For $k \in \NN$, we define the following sets of monomials:
\[
B_k = \{ x^u \colon u \in \ZZ_{\geq 0}^{n+1}, \lvert u \rvert = kd - (n+1) \mbox{ and } u_i < d-1 \mbox{ for all } i \}, 
\]
where $x^u = x_0^{u_0} \ldots x_n^{u_n}$ and $\lvert u \rvert = u_0 + \ldots u_n$. If $x^u$ is contained in some $B_k$, then
we let $k(u)$ denote this $k$. 
Let $\Omega$ denote the $n$-form on $\mathfrak{U}/\mathfrak{S}$ defined by 
\begin{align*}
\Omega = \sum_{i=0}^n (-1)^i x_i d x_0 \wedge \dotsb \wedge \widehat{d x_i} \wedge \dotsb \wedge d x_n,
\end{align*}
where $\widehat{d x_i}$ means that $d x_i$ is left out. We then define
$\mathcal{B}_k = \{ Q \Omega/P^k \colon Q \in B_k \}$, $B=B_1 \cup \ldots \cup B_n$, $\mathcal{B}=\mathcal{B}_1 \cup \ldots \mathcal{B}_n$ and $\lvert B \rvert = b$.
\end{defn}

\begin{thm}
The set $\mathcal{B}$ (restricted to the fibre at $t=0$) forms a basis for $\Hrig^n(U_0)$. Moreover, $\mathcal{B}$ 
also forms a basis for $\HdR^n(\mathfrak{U} \otimes \QQ_q(t))$, so that the same conclusion holds for almost all other fibres. 
\end{thm}

\begin{proof}
\citep[Theorem 3.9]{PT}. 
\end{proof}

\begin{rem}
Note that $b=\frac{1}{d} \bigl((d-1)^{n+1} + (-1)^{n+1}(d-1) \bigr)$ by Theorem \ref{thm:hypersurface}.
\end{rem}

Let $M \in M_{b \times b}(\QQ_q(t))$ denote the matrix of the Gauss--Manin connection $\nabla$ with
respect to the basis $\mathcal{B}=[\omega_1,\ldots,\omega_b]$, i.e. 
\[
\nabla(\omega_j) = \sum_{i=1}^b M_{ij} \omega_i \otimes dt
\]
and let $r \in \ZZ_q[t]$ with $\ord_p(r)=0$ denote a denominator for $M$, i.e. such that we can write $M=G/r$ with 
$G \in M_{b \times b}(\QQ_q[t])$. Moreover, let $\Phi$ denote the matrix of $p^{-1}\Frob_p$ with respect to
the basis $\mathcal{B}$, i.e.
\[
p^{-1} \Frob_p(\omega_j) =\sum_{i=1}^b \Phi_{ij} \omega_i.
\]
Note that $\Phi$ has entries in the ring
\[
\QQ_q \left\langle t, 1/r \right\rangle^{\dag} = 
\Biggl\{\sum_{i,j=0}^{\infty} a_{i,j} \frac{t^i}{r^j} \; : \; 
a_{i,j} \in \QQ_q, \; \exists c > 0 \text{ s.t.}  
\lim_{i+j \rightarrow \infty} \bigl(\ord_p(a_{i,j}) - c(i+j)\bigr) \geq 0
\Biggr\},
\]
of overconvergent functions. 
The fact that $F$ defines a Frobenius structure on the vector bundle $\HdR^n(\mathfrak{U}/\mathfrak{S})$ with its Gauss--Manin 
connection $\nabla$ translates into the following differential equation for the matrix~$\Phi$.

\begin{thm} \label{thm:eqphi} 
The matrix $\Phi \in M_{b \times b}(\QQ_q \langle t, 1/r \rangle^{\dag})$ 
satisfies the differential equation
\begin{align*}
\left(\frac{d}{dt} + \frac{G}{r} \right) \Phi &= p t^{p-1} \Phi \sigma(M), &\Phi(0)& = \Phi_0,
\end{align*}
where $\Phi_0$ is the matrix of $p^{-1}\Frob_p$ on $\Hrig^n(U_0)$ with respect to the 
basis $\mathcal{B}$.
\end{thm}

\begin{proof} 
\citep[Theorem 2.17]{PT}.
\end{proof}

The matrix $\Phi_0$ can be computed using an explicit formula from~\cite{PT} that we will now recall.

\begin{defn}
For $l \in \QQ$ and $r \in \ZZ_{\geq 0}$, let the rising factorial $\prod_{j=0}^{r-1} (l + j)$ be
denoted by $(l)_r$.
\end{defn}

\begin{defn} \label{defn:alpha}
Let $u, v \in \ZZ^{n+1}$ be such that we have
$x^u, x^v \in B$ and 
$p (u_i + 1) \equiv v_i + 1 \bmod{d}$ for all~$i$. 
We define
\begin{equation*}
\alpha_{u,v} = \prod_{i=0}^n a_i^{(p (u_i + 1) - (v_i + 1))/d} 
    \biggl( \sum_{m,r} \left( \frac{u_i+1}{d} \right)_r 
        \sum_{j=0}^{r} \frac{\bigl(p a_i^{p-1}\bigr)^{r-j}}{(m-pj)!j!} \biggr),
\end{equation*}
where the sum in the $i$-th factor of the product is over all integers $m, r \geq 0$  
that satisfy the equation $p(u_i+1)-(v_i+1)=d(m-pr)$.
\end{defn}

\begin{thm} \label{thm:01-03-diagfrob}
Let $\omega_i$ denote an element of $\cB$ corresponding to a tuple 
$u \in \ZZ^{n+1}$ and let $\omega_j$ denote the unique element of~$\cB$ 
corresponding to a tuple $v \in \ZZ^{n+1}$ such that
$p (u_l + 1) \equiv v_l + 1 \bmod{d}$ for all $l$. Then we have
\begin{equation} \label{eq:diagfrobformula}
p^{-1} \Frob_p (\omega_i) = 
    (-1)^{k(v)} \frac{(k(v) - 1)!}{(k(u) - 1)!} p^{n-k(u)} \alpha_{u,v}^{-1} \cdot \omega_j
\end{equation}
as elements of $\Hrig^n(U_0)$.
\end{thm}

\begin{proof}
\citep[Theorem 4.3]{PT}.
\end{proof}

The deformation method consists of the following steps: 

\begin{enumerate}[\it Step 1.]
\item Compute the matrix $M \in M_{b \times b}(\QQ_q(t))$ of the Gauss--Manin connection $\nabla$.
\item Compute the matrix $\Phi_0 \in M_{b \times b}(\QQ_p)$ of the action of the map $p^{-1}\Frob_p$ on 
      $\Hrig^n(U_0)$ using Theorem \ref{thm:01-03-diagfrob}.
\item Solve the differential equation from Theorem~\ref{thm:eqphi} for $\Phi \in M_{b \times b}(\QQ_q \langle t, 1/r \rangle^{\dag})$ and evaluate at $t=1$ to obtain the 
      matrix $\Phi_1 \in M_{b \times b}(\QQ_q) $ of the action of the map $p^{-1}\Frob_p$ on $\Hrig^n(U_1)$.
\item Use Theorem~\ref{thm:hypersurface} to deduce the zeta function $Z(X_1,T)$. 
\end{enumerate}

\begin{rem} \label{rem:generic}
We have already mentioned that the deformation algorithm only applies to generic hypersurfaces. We will
now explain what the exact conditions are~\citep{PT}.

First, the hypersurfaces $\mathcal{X}_0$ and $\mathcal{X}_1$  have 
to be smooth over $\ZZ_q$. For $\mathcal{X}_0$ this is guaranteed if~$p$ does not
divide~$d$ and for $\mathcal{X}_1$ if the hypersurface $X_1$ is smooth. 

Second, we need that the matrix $M$ of the 
Gauss--Manin connection does not have a pole in the open $p$-adic unit disk around $t=1$. Note that such a pole can only be an 
apparent singularity, caused by the generic basis $\mathcal{B}$ not being a basis at the pole. By~\citep[Proposition 3.13]{PT}, 
it is sufficient that $\overline{R}(1) \not = 0$, where $\overline{R} \in \FF_q[t]$ is the reduction modulo~$p$ of the polynomial
$R \in \ZZ_q[t]$ defined in~\citep[Definition 3.12]{PT}. When $X_1$ varies, $\overline{R}(1)$ is a polynomial 
in the coefficients of its defining polynomial. Since $\overline{R}(1) \not = 0$ for diagonal hypersurfaces $X_1$ by the proof 
of~\citep[Theorem 3.6]{PT}, we have that $\overline{R}(1) \not = 0$ for generic $X_1$. 

Therefore, even for fixed $a_0,\ldots,a_n \in \ZZ_p^{\times}$ defining $P_0$, as long as $p$ does not divide~$d$ the deformation
method can be applied to a generic hypersurface $X_1$.  
\end{rem}

\begin{thm} \label{thm:complexityold}
In the deformation method as presented in~\citep{PT}, the four steps above have the following complexities~\citep[\S 7]{PT}: 
\begin{enumerate}[\it Step 1.]
\item \makebox[7cm][l]{time $\SoftOh(\log(p) a^2 (d^{n(\omega+2)} e^{n(\omega+1)} + d^{5n} e^{3n}))$}  \makebox[5cm][l]{space $\SoftOh(\log(p) a^2 d^{4n} e^{3n})$}
\item \makebox[7cm][l]{time $\SoftOh(p a^3 d^{4n})$}                                                   \makebox[5cm][l]{space $\SoftOh(\log(p) a d^{2n})$}
\item \makebox[7cm][l]{time $\SoftOh(p a^3 d^{n(\omega+4)} e^{2n})$}                                   \makebox[5cm][l]{space $\SoftOh(p a^3 d^{5n} e^n)$}
\item \makebox[7cm][l]{time $\SoftOh(\log^2(p) a^2 d^{n(\omega+1)})$}                                  \makebox[5cm][l]{space $\SoftOh(\log(p) a^2 d^{3n})$}
\end{enumerate}
\end{thm}
\begin{proof}
\citep[\S 7]{PT}, where we have taken $a=a'$ and $d_t=1$, because of our choice of $P$.
\end{proof}

Clearly the only
steps not polynomial in $\log(p)$ are the second and the third, both of which are quasilinear in~$p$. So to improve the
dependence of the complexity on~$p$, we only have to consider these steps. All the $p$-adic and $t$-adic precisions 
required to obtain provably correct results remain the same as in~\citep{PT}.

\section{Our new algorithm}

Our strategy to reduce the dependence of the complexity on~$p$ to $\SoftOh(p^{1/2})$ will be the same as in \citep{HarveySqrtp}.
First we reduce the steps in our algorithm that are quasilinear in~$p$ to computing certain matrix products. Then we apply an 
algorithm of Bostan,  Gaudry and Schost~\citep{bgs} (based on~\citep{chud}) to compute these products faster than using the naive 
method. The precise result that we need is as follows.

\begin{thm} \label{thm:bgs}
Let $R$ be a commutative ring and $\textsf{M}$ (respectively $\textsf{MM}$)
a function $\NN \rightarrow \NN$ such that polynomials of degree less than $d$ (respectively matrices of size $m \times m$) can be multiplied
in $\textsf{M}(d)$ (respectively $\textsf{MM}(m)$) ring operations (i.e. $+,-,\times$) in $R$. Moreover, let $A(x)$ be an $m \times m$ 
matrix over the polynomial ring $R[x]$ of degree at most $1$. Suppose that the invertibility conditions of~\citep[Theorem 14]{bgs} are satisfied.
Then for any positive integer $N$, the matrix product $A(1)A(2) \cdots A(N)$ can be computed in
\[
\BigOh\left(\textsf{MM}(m) \sqrt{N} + m^2 \textsf{M}(\sqrt{N})\right) 
\]
ring operations in $R$, storing only $\BigOh(m^2 \sqrt{N})$ elements of $R$.
\end{thm}

\begin{proof} 
This is a special case of~\citep[Theorem 9]{HarveySqrtp}, which is based on \citep[Theorem 15]{bgs}.
\end{proof}

\begin{rem}
We do not explain the invertibility conditions in this theorem in detail. However, they are
easily seen to be satisfied when $2,3,\dots,2^{s}+1 \in R^{\times}$ where $s=\lfloor \log_4(N) \rfloor$. In the 
case we are interested in ($R=\ZZ_p$ or $\ZZ_q$ and $N=p$) this is the case for all $p>2$. For $p=2$ the results in 
this paper do not improve the ones from~\citep{PT} anyway. \\
\end{rem}

\subsection{Diagonal fibre} \mbox{ } \\

We start by analyzing Step 2 of the deformation method, i.e. the computation of the matrix $\Phi_0$.
In~\citep[\S 4]{PT} it is explained that we have to compute the
\begin{equation} \label{eq:alpha}
\alpha_{u,v} = \prod_{i=0}^n a_i^{(p (u_i + 1) - (v_i + 1))/d} 
    \biggl( \sum_{m,r} \left( \frac{u_i+1}{d} \right)_r 
        \sum_{j=0}^{r} \frac{\bigl(p a_i^{p-1}\bigr)^{r-j}}{(m-pj)!j!} \biggr),
\end{equation}
to $p$-adic precision $N_{\Phi_0}' \in \SoftOh(ad^n)$ and that for this we only have to consider terms in the outer sum with 
$r \leq \mathcal{R}, m \leq \mathcal{M} \mbox{ and } p(u_i+1)-(v_i+1)=d(m-pr)$,
where $\mathcal{R} \in \SoftOh(ad^n)$ and $\mathcal{M} \in \BigOh(p\mathcal{R})$. Note that all $p$-adic precisions in the algorithm are $\SoftOh(ad^n)$ by~\citep[\S 7]{PT}, 
so that a single multiplication in $\QQ_p$ takes time and space $\SoftOh(\log(p)ad^n)$.

From the formula above we see that we have to compute $k!$ (to finite $p$-adic precision) for all $i=0, \ldots, n$ and all $k \leq \mathcal{M}$ such that 
\[
k \equiv \frac{p(u_i+1)-(v_i+1)}{d} \bmod{p}. 
\]
Computing these factorials naively as in~\citep{PT} clearly takes time at least linear in~$p$, so we will have to proceed
differently.

Let us fix a basis vector $u \in B$ and let $v \in B$ be the unique basis vector such that we have $p(u_i+1)-(v_i+1) \equiv 0 \bmod{d}$ for all $0 \leq i \leq n$.
For a single value of $0 \leq i \leq n$ we can compute and store all $k!$ with $k \leq \mathcal{M}$ and 
\[k \equiv \frac{p(u_i+1)-(v_i+1)}{d} \bmod{p}\]
in time and space $\SoftOh(p^{1/2} a^2 d^{2n})$. Indeed, we need to compute $k!$ for $\mathcal{R} \in \SoftOh(ad^n)$ different values
of $k$ and to go from one value of $k$ to the next we have to multiply by the product 
\[(k+1) \ldots (k+p),\]
which can be computed in time and space $\SoftOh(p^{1/2}ad^n)$ using the algorithm from Theorem~\ref{thm:bgs} with $m=1$, $R=\ZZ_p$, $A(x)=(x+k)$ and $N=p$. 
Doing this computation for all $0 \leq i \leq n$ instead of a single value of $i$ does not change the time and space requirements, since the factor $n$ is 
absorbed by the $\SoftOh$ symbol. This takes care of all the $(m-pj)!$ we have to compute in~\eqref{eq:alpha}.

The elements
\[ 
a_i^{(p (u_i + 1) - (v_i + 1))/d}, \left( \frac{u_i+1}{d} \right)_r, (pa_i^{p-1})^{r-j} \mbox{ and } j!
\]
that we need in~\eqref{eq:alpha} (note that $u$ is still fixed), can be computed and stored naively in time and space
$$\SoftOh(n \mathcal{R} \log^2(p) ad^n) \subset \SoftOh(\log^2(p) a^2 d^{2n}).$$ 
Moreover, $\alpha_{u,v}$ can be computed 
from these elements and the $k!$ in time 
$$\SoftOh(n \mathcal{R}^2 \log(p) a d^n) \subset \SoftOh(\log(p) a^3 d^{3n})$$
and additional space $\SoftOh(\log(p) a d^n)$. Putting everything together, we find that a single $\alpha_{u,v}$ can 
be computed in time $\SoftOh(p^{1/2} a^2 d^{2n}+a^3 d^{3n})$ and space $\SoftOh(p^{1/2} a^2 d^{2n})$. Since
there are $b \in \BigOh(d^n)$ different basis vectors $u$, all of the $\alpha_{u,v}$ can be computed in time 
$\SoftOh(p^{1/2} a^2 d^{3n}+a^3 d^{4n})$ and space $\SoftOh(p^{1/2} a^2 d^{2n})$. The same is then clearly true
for the matrix $\Phi_0$ by Theorem~\ref{thm:01-03-diagfrob}.

Summarising, we have proved the following theorem:
\begin{thm} \label{thm:complexitystep2}
Step 2 of the deformation method, i.e. the computation of the matrix $\Phi_0$ to the required $p$-adic precision can be carried out in
\begin{align*}
&\mbox{ time: } \; \; \; \SoftOh(p^{1/2} a^2 d^{3n}+a^3 d^{4n}), &\mbox{ space: }& \; \; \; \SoftOh(p^{1/2} a^2 d^{2n}).
\end{align*}

\end{thm}

\subsection{Differential equation} \mbox{ } \\

We now move on to the third step of the deformation method, i.e. the computation of the matrix $\Phi_1$
by solving the differential equation from Theorem~\ref{thm:eqphi} for $\Phi$ and evaluating this matrix at $t=1$. 
Let us start by recalling what we did in~\citep{PT}, to understand why that is quasilinear in~$p$ and how it can be improved. 

Recall from~\citep[\S 5]{PT} that if $C \in M_{b \times b}(\QQ_q[[t]])$ denotes a fundamental matrix of horizontal sections of the
Gauss--Manin connection $\nabla$ with respect to the basis $\mathcal{B}$, i.e. such that
\begin{align*}
&\left(\frac{d}{dt}+\frac{G}{r} \right) C = 0,            &C(0)=I&, 
\end{align*}
then $\Phi = C \Phi_0 \sigma(C^{-1})$ is a solution to the differential equation from Theorem~\ref{thm:eqphi}. Let us write 
$G=\sum_{i=0}^{\deg(G)} G_i t^i$ and $r=\sum_{i=0}^{\deg(r)} r_i t^i$. Note that $\deg(G),\deg(r) \in \SoftOh((de)^n)$ by
\citep[\S 7]{PT}. A power series
solution $C=\sum_{i=0}^{\infty} C_i t^i$ can be obtained by solving the following matrix recurrence~\cite[(5.5)]{PT}:
\begin{align*}
  C_0&=I, \\
  C_{i+1} &= \frac{-1}{r_0(i+1)} \left( \sum_{j=i-\deg(G)}^i G_{i-j} C_j + \sum_{j=i-\deg(r)+1}^i r_{i-j+1}(jC_j) \right),
\end{align*}
where we take $C_j=0$ for $j < 0$. 

The inverse matrix $C^{-1}$ satisfies the 
dual differential equation~\citep[Remark 5.7]{PT}:
\begin{align*}
&\left(\frac{d}{dt}-\frac{G^t}{r} \right) (C^{-1})^t = 0,            &C^{-1}(0)=I&,
\end{align*}
which can be solved by a similar recurrence.

To obtain $\Phi$ as an element of $M_{b \times b}(\QQ_q[[t]])$ to $t$-adic precision~$K$, we need to determine
$C_i$ for $i < K$ and $(C^{-1})_i$ for $i < \lceil K/p \rceil$. Then we compute the product 
$$\Phi = C \Phi_0 \sigma(C^{-1}),$$ 
where $\Phi_0$ has already been obtained in the second step. It is shown in~\citep[\S 7]{PT}
that to recover $\Phi_1$ to the required $p$-adic precision, we have to take 
$$K \in \SoftOh(pad^{2n}e^n).$$ 
Because of some technical convergence issues, we need to multiply $\Phi$ by the polynomial~$s \in \ZZ_q[t]$ from~\citep[Theorem 6.6]{PT}
and truncate it to $t$-adic precision~$K$ again before we can evaluate it. By~\citep[\S 7]{PT} the degree of $s$ is in $\SoftOh(pad^{2n}e^n)$. 
Finally, we evaluate $(s/s(1))\Phi$ at $t=1$ to obtain the matrix $\Phi_1$. Exact bounds for all $p$-adic and $t$-adic 
precisions can be found in~\citep[\S 5, \S 6]{PT}. All $p$-adic precisions in the algorithm are $\SoftOh(ad^n)$, so elements of $\QQ_q$ can be multiplied
in time and space~$\SoftOh(\log(p) a^2 d^{n})$.

The computation from the previous paragraph takes time at least linear in~$p$: the number of terms of $C$ is $K$ so computing this object or 
multiplying by it already has complexity at least linear in~$p$. Note that we do not have this problem for the matrix $C^{-1}$ since that
has $\lceil K/p \rceil$ terms which is $\SoftOh(ad^{2n}e^n)$ by Theorem~\ref{thm:extra} below. Recall that $\sigma$ sends $t$ to $t^p$, 
so that the number of nonzero terms in $\sigma(C^{-1})$ to $t$-adic precision $K$ is still $\SoftOh(ad^{2n}e^n)$. Finally, the evaluation of 
$(s/s(1))\Phi$ at $t=1$ again has complexity at least linear in~$p$, because there are $K$ terms. Hence we will have to proceed differently.

We may assume $s=r^{\theta}$ with $\theta/p \in \SoftOh(ad^n)$, by Theorem~\ref{thm:extra} below. We have to compute:
\[
(s/s(1))\Phi = (r/r(1))^{\theta} C \Phi_0 \sigma(C^{-1}) 
\]
as an element of $M_{b \times b}(\QQ_q[[t]])$ to $t$-adic precision $K$ and then evaluate at $t=1$. 
We know that the multiplication of $(r/r(1))^{\theta}$ and $C$ has complexity at 
least linear in~$p$. However, as already observed by Hubrechts~\citep{Hubrechts2011}[\S 3.1], we can avoid this multiplication by solving a slightly different differential equation instead. Let 
$D \in M_{b \times b}(\QQ_q[[t]])$ denote the matrix $D=(r/r(1))^{\theta}C$. Then $D=\sum_{i=0}^{\infty} D_i t^i$ satisfies the differential 
equation
\begin{align*}
&\left(\frac{d}{dt} + \frac{G-\theta \frac{dr}{dt} I}{r}\right)D = 0,  &D(0) = \left(\frac{r(0)}{r(1)}\right)^{\theta} I&. 
\end{align*}
So we can avoid the multiplication by $(r/r(1))^{\theta}$ simply by replacing the matrix $G$ by the matrix 
\[H=G-\theta \frac{dr}{dt} I\] 
and solving the recurrence
\begin{align*}
  D_0&=\left(\frac{r(0)}{r(1)}\right)^{\theta} I, \\
  D_{i+1} &= \frac{-1}{r_0(i+1)} \left( \sum_{j=i-\deg(H)}^i H_{i-j} D_j + \sum_{j=i-\deg(r)+1}^i r_{i-j+1}(jD_j) \right),
\end{align*}
where we take $D_j=0$ for $j < 0$. 

The matrix $\Phi_1$ is now given by the evaluation of $D \Phi_0 \sigma(C^{-1})$ to $t$-adic precision~$K$ at $t=1$, which we can
also write as
\begin{align*} 
\Phi_1 = \sum_{j=0}^{\lceil K/p \rceil-1} \left( \sum_{i=0}^{K-1-pj} D_i \right) \cdot \Phi_0 \cdot \sigma((C^{-1})_j). 
\end{align*}
The number of terms in the inner sum (say for $j$=0) is at least linear in~$p$, so evaluating it will have complexity at least linear
in~$p$. However, since the $D_i$ satisfy a recurrence, so do the sums $E_i = D_0+\ldots+D_i$:
\begin{align*}
  E_0&=\left(\frac{r(0)}{r(1)}\right)^{\theta} I, \\
  E_{i+1} &= E_i+\frac{-1}{r_0(i+1)} \left( \sum_{j=i-\deg(H)}^i H_{i-j} (E_j-E_{j-1}) + \sum_{j=i-\deg(r)+1}^i r_{i-j+1} \cdot j \cdot (E_j-E_{j-1}) \right),
\end{align*}
where we take $E_j=0$ for $j < 0$. 

The expression for the matrix $\Phi_1$ becomes
\begin{align} \label{eq:Phi12}
\Phi_1 = \sum_{j=0}^{\lceil K/p \rceil-1} E_{K-1-pj} \cdot \Phi_0 \cdot \sigma((C^{-1})_j). 
\end{align}

Note that the orders of the recurrences for $C_i,(C^{-1})_i,D_i,E_i$ are all $\SoftOh((de)^n)$ since $\deg(G)$ and $\deg(r)$ are $\SoftOh((de)^n)$. 
However, to apply Theorem~\ref{thm:bgs}, we need a recurrence of order~$1$. 
The order of the recurrence for $E_i$ is given by
\[\kappa = \max \{ \deg(H)+1, \deg(r)\}+1.\]
Let $\mathbb{E}_i \in M_{\kappa b \times b}(\QQ_q)$ be the matrix with vertical blocks $E_i,E_{i-1},\ldots,E_{i-\kappa+1}$. 
Then from the recurrence for $E_i$, we find a matrix $A(x) \in M_{b \times b}(\QQ_q[x])$ of degree~$1$ such that
\[
\mathbb{E}_{i+1} = (i+1)^{-1} A(i+1) \mathbb{E}_i 
\]
for all $i \geq 0$. Note that $E_i$ can be easily read off from $\mathbb{E}_i$.

Recall that $\lceil K/p \rceil \in \SoftOh(ad^{2n}e^n)$. Hence the $(C^{-1})_i$ with 
$i \leq \lceil K/p \rceil$ can be computed and stored in  
\begin{align*}
&\mbox{time: } \SoftOh((K/p) (de)^n b^{\omega} a^2 d^n) \subset \SoftOh(\log(p) a^3 d^{n(\omega+4)} e^{2n}), \\
&\mbox{space: }\SoftOh((K/p) b^2 a^2 d^n) \subset \SoftOh(\log(p)a^3 d^{5n} e^n). 
\end{align*}
Applying $\sigma$ to $C^{-1}$ takes time 
$$\SoftOh((K/p) b^2 (\log^2(p) a + \log(p) a^2 d^n)) \subset \SoftOh(\log^2(p) a^3 d^{5n} e^n)$$ 
and negigible additional space. This is all the same as in~\citep[Proposition 7.7]{PT}. 
Now to compute $\Phi_1$, we can proceed in two different ways. 

The first way is to evaluate~\eqref{eq:Phi12} by solving the recurrence for $E_i$ naively. 
Note that we only have to store the last $\SoftOh((de)^n)$ matrices $E_i$ as already
observed by Hubrechts~\citep{Hubrechts2011}[Theorem 2]. Therefore, this takes 
\begin{align*}
&\mbox{time: } \SoftOh(K (de)^n b^{\omega} a^2d^n) \subset \SoftOh(pa^3d^{n(\omega+4)} e^{2n}), \\ 
&\mbox{space: } \SoftOh((de)^n b^2 \log(p) a^2 d^n) \subset \SoftOh(\log(p) a^2 d^{4n} e^n).
\end{align*}

The second way is to determine $\mathbb{E}_{i+p}$ from $\mathbb{E}_i$ repeatedly, using Theorem~\ref{thm:bgs} to compute
\begin{align*}
A(i+p) A(i+p-1) \ldots A(i+1) \;\;\;\;\; \mbox{ and } \;\;\;\;\; (i+1)(i+2) \ldots (i+p). 
\end{align*}
That the indices in the first product are decreasing instead of increasing can easily be circumvented by transposing. We may
assume without loss of generality that the matrix $A(x)$ has entries in $\ZZ_q[x]$, since by~\citep{PT}[Remark 3.11] this is 
the case when $p>n$ and the results in this paper do not improve the ones from~\citep{PT} when $p \leq n$ anyway. We will 
have to increase the $p$-adic precision by 
$$\ord_p((i+1)(i+2)\ldots(i+p)) \in \BigOh(\log_p(K)),$$ 
but this remains $\SoftOh(ad^n)$. Therefore $\mathbb{E}_{i+p}$ can be determined from $\mathbb{E}_i$ in time and space
$$\SoftOh(p^{1/2} (\kappa b)^{\omega} a^2 d^n) \subset \SoftOh(p^{1/2}a^2 d^{n(2\omega+1)} e^{n\omega}).$$
We need to do this $\lceil K/p \rceil$ times and update $\Phi_1$ each time. This takes time
$$\SoftOh\left((K/p) \left(p^{1/2}a^2 d^{n(2\omega+1)} e^{n\omega} + b^{\omega} \log(p) a^2 d^n \right)\right) \subset \SoftOh\left(p^{1/2}a^3 d^{n(2\omega+3)} e^{n(\omega+1)}\right)$$
and negigible additional space. \\

In the discussion above we have used two results that are implicit in~\citep[\S 7]{PT}, but cannot be found there explicitly. For completeness we prove
these results in the next theorem.

\begin{thm} \label{thm:extra} We can take:
\begin{enumerate}
\item $s=r^{\theta}$ with $\theta/p \in \SoftOh(ad^n)$ and $\deg(r) \in \SoftOh((de)^n)$, 
\item $K/p \in \SoftOh(ad^{2n}e^n)$.
\end{enumerate}
\end{thm}

\begin{proof}
We can take $r$ equal to the polynomial $R=\prod_{i=2}^{n+1} \det(\Delta_k)$ from~\citep[Definition 3.12]{PT} which has degree
$\SoftOh((de)^n)$. Let $N_{\Phi}$ 
denote the $p$-adic precision to which the matrix $\Phi \in M_{b \times b}(\QQ_q \langle t,1/r \rangle)$ has to be computed. 
From~\citep[\S 7]{PT}, we know that $N_{\Phi} \in \SoftOh(ad^n)$. By~\citep[Theorems 6.4 (with $z \neq \infty$)]{PT} 
we can take
\[
s = \det(\Delta_2 \ldots \Delta_n) \det(\Delta_{n+1})^{p(n+h(N_{\Phi}))-n},  
\]
where
\[
h(N_{\Phi}) = \max \left\{ i \in \NN \colon i+(n-1)+\ord_p((n-1)!-n \lfloor \log_p(p(n+i)-n) \rfloor < N_{\Phi} \right\}. 
\]
We can take $\theta=p(n+h(N_{\Phi}))$, so that $\theta/p \in \SoftOh(ad^n)$ (note that the logarithm in the definition of 
$h(N_{\Phi})$ is to base~$p$). By~\citep[Theorem 6.4 (with $z = \infty$)]{PT} it follows that 
\[
K \leq \deg(s)+ 1+\deg \Bigl(\prod_{k=2}^n \Delta_k^{-1}\Bigr) + (p(n+h(N_{\Phi}))-n) \deg(\Delta_{n+1}^{-1}) + ph(N_{\Phi}),
\]
where $\deg$ denotes minus the order at $z=\infty$, so that $K/p \in \SoftOh(ad^{2n}e^n)$.
\end{proof}

Summarising, we have proved the following theorem.

\begin{thm} \label{thm:complexitystep3}
Step 3 of the deformation method, i.e. the computation of the matrix $\Phi_1$ to the required $p$-adic precision, can be carried out in either
\begin{align*}
&\mbox{time: } \SoftOh\left(p a^3 d^{n(\omega+4)} e^{2n}\right),                 &\mbox{space: }& \SoftOh\left(\log(p) a^3 d^{5n} e^n\right), \\
\intertext{or alternatively}
&\mbox{time: } \SoftOh\left(p^{1/2} a^3 d^{n(2\omega+3)} e^{n(\omega+1)}\right), &\mbox{space: }& \SoftOh\left(p^{1/2} a^2 d^{n(2\omega+1)} e^{n\omega}+a^3d^{5n}e^n\right). 
\end{align*}
\end{thm}
\begin{proof}
This follows by adding up all complexities from the discussion above, leaving out terms that are dominated by other terms or powers of logarithms of other terms.
The two different bounds correspond to the two different ways of evaluating $\Phi_1$ explained above. 
\end{proof}

\begin{rem}
Note that if $K/p \in \BigOh(K^{1/2-\epsilon})$ for some $0 < \epsilon < 1/2$, then by~\citep[Theorem 9]{HarveySqrtp} (which
is slightly stronger than Theorem $\ref{thm:bgs}$), all the way at the end of the $\SoftOh(p^{1/2})$ algorithm we can save a 
factor $K/p$ and compute $\Phi_1$ to the required $p$-adic precision in 
\begin{align*}
&\mbox{time: }  \SoftOh\left(p^{1/2}a^2 d^{n(2\omega+1)} e^{n\omega} + a^3 d^{n(\omega+3)} e^n\right),
&\mbox{space: }& \SoftOh\left(p^{1/2}a^2 d^{n(2\omega+1)} e^{n\omega} +a^3d^{5n}e^n \right).
\end{align*}
This improvement only applies when~$p$ is large enough relative to $a$ and $d^n$. Since the exact condition on~$p$ is 
rather complicated, we will not use this improvement in the theorem below.
\end{rem}

We now put everything together to obtain our main result.

\begin{thm} \label{thm:complexity}
Let $X$ be a generic hypersurface of degree $d$ in projective space $\PP^n_{\FF_q}$ over
a finite field $\FF_q$ of characteristic~$p$ not dividing $d$ and cardinality $q=p^a$. Then the 
zeta function $Z(X,T)$ may be computed in
\begin{align*}
&\mbox{time: }  \SoftOh\left(p^{1/2} a^3 d^{n(2\omega+3)} e^{n(\omega+1)}\right), \\
&\mbox{space: } \SoftOh\left(p^{1/2} a^2 d^{n(2\omega+1)}e^{n \omega}+a^3 d^{5n} e^n + a^2 d^{4n} e^{3n} \right).
\end{align*} 
\end{thm}

\begin{proof}
This follows from Theorems \ref{thm:complexityold}, \ref{thm:complexitystep2} and the second part of Theorem~\ref{thm:complexitystep3}. 
\end{proof}

Moreover, we obtain a space efficient version of the deformation method as well.

\begin{thm} \label{thm:complexity2}
Let $X$ be a generic hypersurface of degree $d$ in projective space $\PP^n_{\FF_q}$ over
a finite field $\FF_q$ of characteristic~$p$ not dividing $d$ and cardinality $q=p^a$. Then the 
zeta function $Z(X,T)$ may be computed in
\begin{align*}
&\mbox{time: }  \SoftOh\left(pa^3d^{n(\omega+4)}e^{2n}+a^2(d^{n(\omega+2)}e^{n(\omega+1)}+d^{5n} e^{3n})\right), \\
&\mbox{space: } \SoftOh\left(\log(p) \left(a^3 d^{5n} e^n + a^2 d^{4n} e^{3n}\right) \right).
\end{align*} 
\end{thm}

\begin{proof}
This follows from Theorems \ref{thm:complexityold}, \ref{thm:complexitystep2} and the first part of Theorem~\ref{thm:complexitystep3}. 
\end{proof}

\begin{rem} 
Since this result is obtained by a straightforward combination of the work in~\citep{PT} and~\citep{Hubrechts2011}, it should not come
as a surprise to experts. However, as far as we are aware, no algorithm with similar time and space complexities appears in the literature. 
In particular, this algorithm represents an improvement of another algorithm of Harvey~\citep{HarveyArithmeticSchemes} which runs in
\begin{align*}
\mbox{time: }  &\BigOh \left(p \log(p)^{1+\epsilon} 2^{6n^2+13n}  a^{3n+4+\epsilon} n^{3n+4+\epsilon} (d+1)^{3n^2+6n+\epsilon} \right), \\ 
\mbox{space: } &\BigOh\left(\log(p) 2^{4n^2+9n}  a^{2n+3} n^{2n+2} d^{2n^2+4n} \right),
\end{align*}
but again can be applied more generally and is conceptually simpler.
\end{rem}

\begin{rem}
Most ideas in this paper and~\citep{PT} are not limited to smooth projective projective hypersurfaces, but
also apply to nondegenerate hypersurfaces in projective toric varieties for example. The only thing that is not straightforward is to find a replacement 
for the class of diagonal projective hypersurfaces, for which it is easy to compute $\Phi_0$. A good guess would be 
to take hypersurfaces with defining polynomials that are maximally sparse or which have a lot of automorphisms.
\end{rem}

\bibliographystyle{plainnat}
\bibliography{deformation2}

\end{document}